\documentclass[pdflatex,sn-mathphys-num]{sn-jnl}

\usepackage{graphicx}%
\usepackage{multirow}%
\usepackage{amsmath,amssymb,amsfonts}%
\usepackage{amsthm}%
\usepackage{mathrsfs}%
\usepackage[title]{appendix}%
\usepackage{xcolor}%
\usepackage{textcomp}%

\usepackage{booktabs}%
\usepackage{algorithm}%
\usepackage{algorithmicx}%
\usepackage{algpseudocode}%
\usepackage{listings}%

\theoremstyle{thmstyleone}%
\newtheorem{theorem}{Theorem}[section]
\newtheorem{proposition}[theorem]{Proposition}%
\newtheorem{corollary}[theorem]{Corollary}%

\theoremstyle{thmstyletwo}%

\theoremstyle{thmstylethree}%

\begin{document}

\title[General tropical convergence of harmonic amoebas]
{General tropical convergence of harmonic amoebas
}


\author{\fnm{Takashi} \sur{Ichikawa}}\email{ichikawn@cc.saga-u.ac.jp}




\affil{\orgdiv{Department of Mathematics}, 
\orgname{Faculty of Science and Engineering}, 
\orgaddress{\street{Saga 840-8502}, \country{Japan}}}





\abstract{
By using Schottky uniformization theory of degenerating algebraic curves, 
we describe the tropical convergence of harmonic amoebas of pointed Riemann surfaces to 
tropical curves which are not necessarily simple. 
We extend Lang's results on the simple tropical convergence 
based on the Frenchel-Nielsen coordinates to the nonsimple case. 
Our results are hoped to give contributions in compactifying the moduli space of 
pointed Riemann surfaces with tropical curves, 
and in studying crystallization of general dimer models. 
}




\maketitle

\section{Introduction}

Amoebas of complex varieties and their tropical convergence were main research subjects connecting 
complex geometry with tropical geometry, 
and amoebas of Schottky uniformized pointed Riemann surfaces were recently used for 
studying dimer models in statistical mechanics (cf. \cite{BB1, BBS}). 
Lang \cite{L} generalized the tropical convergence shown by Mikhalkin \cite{M1, M2} 
for harmonic amoebas  expressed as integrals of imaginary normalized differentials 
on pointed Riemann surfaces, 
and Berggren-Borodin \cite{BB2} applied his results to studying 
crystallization of the Aztec diamond dimer models. 
Lang used the Frenchel-Nielsen coordinates on the Teichm\"{u}ller space 
associated with pants decomposition, 
and hence he treat only the case that tropical curves are {\it simple}, 
namely, the underlying graphs are trivalent. 
Therefore, 
it is hoped to extend Lang's results to general tropical curves which are not 
necessarily simple for compactifying the moduli space of 
pointed Riemann surfaces with tropical curves (see comments in \cite[Section 1.3]{L}), 
and studying crystallization of general dimer models 
(see comments on \cite[Theorem 1.1]{BB2}). 
The aim of the present paper is to apply the Schottky uniformization theory on families of 
algebraic curves with possibly reducible degeneration \cite{I1, I2} to studying 
the tropical and phase-tropical convergences of harmonic amoebas 
of pointed Riemann surfaces to general tropical curves. 
By using explicit formulas of imaginary normalized differentials and their integrals, 
we show the following main results of this paper in a straight way. 

Let $C = (G, \ell)$ be a tropical curve of genus $g$ with $n$ leaves 
provided that $2g + n > 2$. 
Therefore, $G$ is a connected graph consisting of the sets $V_{G}$ of vertices, 
$E_{G}$ of oriented edges and $L_{G}$ of leaves 
such that ${\rm rank} \, H_{1}(G, \mathbb{Z}) = g$, 
$|L_{G}| = n$, 
and $\ell : E_{G} \rightarrow {\mathbb R}_{> 0}$ is a length function. 
Then for any $m \times n$ real residue matrix $R = (R_{1},..., R_{m})^{T}$, 
there exists a unique set $\{ {\rm w}_{R_{i}} = {\rm w}_{R_{i}}^{C} \}_{1 \leq i \leq m}$ 
of the associated $1$-forms on $C$ 
which are {\it exact}, namely, their integrals along any loop in $G$ are $0$. 
Therefore, 
one has a map $\pi_{R} : C \rightarrow {\mathbb R}^{m}$ defined as 
$$
\pi_{R}(q) = \left( \int_{p}^{q} {\rm w}_{R_{1}},..., \int_{p}^{q} {\rm w}_{R_{m}} \right) \quad (q \in C)
$$
up to the choice of a point $p \in C$, 
and the set $\pi_{R}(C) \subset {\mathbb R}^{m}$ is called a {\it harmonic tropical curve}. 
Furthermore, 
take a family $\{ \mathcal{R}_{s} \}_{s > 0}$ of Riemann surfaces of genus $g$ with $n$ 
marked points which are deformations of a singular complex curve $C_{0}$ with dual graph $G$. 
Then for each $\mathcal{R}_{s}$, 
there exists a unique set 
$\{ \omega_{R_{i}} = \omega_{R_{i}}^{\mathcal{R}_{s}} \}_{1 \leq i \leq m}$ of 
the (Abelian) differentials on $\mathcal{R}_{s}$ associated with $R$ 
which are {\it imaginary normalized}, 
namely, their integrals along any loop in $\mathcal{R}_{s}$ are purely imaginary. 
Therefore, 
one has a map $\mathcal{A}_{R} : \mathcal{R}_{s} \rightarrow {\mathbb R}^{m}$ defined as 
$$
\mathcal{A}_{R}(q) = \left( {\rm Re} \left( \int_{p}^{q} \omega_{R_{1}} \right),..., 
{\rm Re} \left( \int_{p}^{q} \omega_{R_{m}} \right) \right) \quad (q \in \mathcal{R}_{s})
$$
up to the choice of a point $p \in \mathcal{R}_{s}$, 
and the set $\mathcal{A}_{R}(\mathcal{R}_{s}) \subset {\mathbb R}^{m}$ is called 
a {\it harmonic amoeba}. 

\begin{theorem} {\rm (cf. \cite[Theorem 1]{L} for simple tropical curves)} 
Let $\{ \mathcal{R}_{s} \}_{s>0}$ be a family of $n$-pointed Riemann surfaces of genus $g$ 
as deformations of $C_{0}$ whose deformation parameters $y_{e}$ $(e \in E_{G})$ 
defined in (3.1) below satisfy $|y_{e}| = s^{\ell_{s}(e)}$, 
where $\lim_{s \downarrow 0} \ell_{s}(e) = \ell(e)$. 
Then under $s \downarrow 0$, 
the sequence of maps 
$$
\frac{1}{\log s} \mathcal{A}_{R} : \mathcal{R}_{s} \rightarrow {\mathbb R}^{m}
$$ 
converges to the map $\pi_{R} : C \rightarrow {\mathbb R}^{m}$ 
with respect to the Hausdorff distance on compact sets. 
\end{theorem}

\begin{theorem} {\rm (cf. \cite[Theorem 4]{L} for simple tropical curves)} 
Let $\{ \mathcal{R}_{s} \}_{s>0}$ be the above deformations of $C_{0}$ 
whose deformation parameters $y_{e}$ are given by $y_{e} = s^{\ell_{s}(e)} \cdot \theta_{s}(e)$ 
with $(\ell_{s}(e), \theta_{s}(e)) \in \mathbb{R}_{>0} \times S^{1}$. 
If $\lim_{s \downarrow 0} \ell_{s}(e) = \ell(e)$, 
and $\lim_{s \downarrow 0} \theta_{s}(e)$ exists as an element $\theta(e)$ of $S^{1}$, 
then for any $i = 1,..., m$ and oriented loop $\rho$ in $G$ given by the composition 
$e_{1} \cdots e_{n}$ of oriented edges, 
one has 
$$
\lim_{s \downarrow 0} \oint_{\rho_{s}} \omega_{R_{i}}^{\mathcal{R}_{s}} 
= \sum_{j=1}^{n} \log (\theta(e_{j})) \, {\rm w}_{R_{i}}^{C}(e_{j}), 
$$
where $\rho_{s}$ denotes the loop in $\mathcal{R}_{s}$ associated with $\rho$, 
and the branch of $\log$ is taken such that 
$\log: S^{1} \rightarrow [0, 2 \pi \sqrt{-1}) \subset \mathbb{C}$. 
\end{theorem}

We hope that calculations in \cite[Section 2 and Proposition 3.1]{I2} 
will give precise evaluations of the above limit formulas.

\section{$1$-forms on tropical curves} 

We review results of \cite[Section 2.5]{L}. 
Let $C = (G, \ell)$ be a tropical curve of genus $g$ with $n$ leaves as above, 
and for each $h \in LE_{G} := L_{G} \cup E_{G}$, 
$v(h)$ denotes the boundary (resp. target) vertex of $h$ if $h \in L_{G}$ (resp. $h \in E_{G}$). 
A (tropical) $1$-form ${\rm w}$ on $C$ is a real-valued function on $LE_{G}$ 
such that for any $e \in E_{G}$, ${\rm w}(e) + {\rm w}(-e) = 0$, 
and that for any $v \in V_{G}$, 
all elements $h_{1},..., h_{n}$ of $LE_{G}$ with $v(h_{j}) = v$ satisfy 
$$
\sum_{j=1}^{n} {\rm w}(h_{j}) = 0.
$$
Denote by $\Omega(C)$ the vector space of $1$-forms on $C$, 
and for any leaf $l$ of $C$, 
define the residue of ${\rm w} \in \Omega(C)$ at $l$ to be the number ${\rm w}(l)$. 
Then an element ${\rm w} \in \Omega(C)$ is holomorphic if all its residues are zero, 
and denote by $\Omega_{\mathcal{H}}(C)$ 
the vector space of holomorphic $1$-forms on $C$. 

A path on a tropical curve $C = (G, \ell)$ is a continuous map $\rho: [0, 1] \rightarrow LE_{G}$ 
such that $\rho(0), \rho(1) \in V_{G}$ and that the restriction of $\rho$ to $[0, 1)$ is injective. 
In particular, 
a path can join a vertex of $G$ to an other, 
and a loop is a path $\rho$ satisfying $\rho(0) = \rho(1)$. 
Therefore, 
one can represent a path by the composition of ordered collection $h_{1} \cdots h_{n}$ 
of elements of $LE_{G}$ such that $v(h_{i}) = v(-h_{i+1})$, 
and for a path $\rho = h_{1} \cdots h_{n}$ and any $1$-form ${\rm w}$ on $C$, 
we define the integral
$$
\int_{\rho} {\rm w} := \sum_{j=1}^{n} \ell(h_{j}) {\rm w}(h_{j}). 
$$ 
A $1$-form ${\rm w} \in \Omega(C)$ is exact if $\int_{\rho} {\rm w} = 0$ 
for any loop $\rho$ in $G$, 
and the vector space of exact $1$-forms on $C$ is denoted by $\Omega_{0}(C)$. 
Furthermore, 
for a path from a leaf to another or a loop $\rho = h_{1} \cdots h_{n}$ in $C$, 
define the $1$-form ${\rm w}^{\rho} \in \Omega(C)$ dual to $\rho$ by
$$
{\rm w}^{\rho}(h) = \left\{ \begin{array}{ll} 
1 & \text{if $h \in {\rm Im}(\rho) := \{ h_{1},..., h_{n} \}$}, 
\\
-1 & \text{if $-h \in {\rm Im}(\rho)$}, 
\\
0 & \text{otherwise}. 
\end{array} \right. 
$$
Then it is shown in \cite[Proposition 2.27]{L} that 
if $C$ is a tropical curve of genus $g$ with $n \geq 2$ leaves, then 
\begin{itemize}

\item 
the sum of the residues of any $1$-form $w \in \Omega(C)$ is zero, 

\item 
$\dim_{\mathbb{R}} \Omega_{0}(C) = n - 1$, 
$\dim_{\mathbb{R}} \Omega_{\mathcal{H}}(C) = g$ and 
$\Omega(C) = \Omega_{0}(C) \oplus \Omega_{\mathcal{H}}(C)$, 

\item 
any element of $\Omega_{0}(C)$ is determined by its residues.

\end{itemize}
For each real vector $R = (r_{l})_{l \in L}$ with $|L| = n$ entries such that 
$\sum_{l \in L} r_{l} = 0$, 
denote by ${\rm w}_{R}$ the unique exact $1$-form on $C$ whose residue vector is $R$, 
namely, ${\rm w}_{R}(l) = r_{l}$ for all $l \in L$.

\section{Schottky uniformization and Abelian differentials}

We review results of \cite[Sections 2 and 3]{I2}. 

\subsection{Schottky uniformized deformations}

Let $G = (V_{G}, E_{G}, L_{G})$ be a connected graph. 
We attach variables $x_{h}$ for $h \in LE_{G}$, 
and $y_{e} = y_{-e}$ for each $e \in E_{G}$. 
Denote by $R_{G}$ the ${\mathbb Z}$-algebra generated by $x_{h}$ $(h \in LE_{G})$, 
$(x_{e} - x_{-e})^{-1}$ $(e \in E_{G})$ and 
$(x_{h} - x_{h'})^{-1}$ ($h, h' \in LE_{G}$ with $h \neq h'$, $v(h) = v(h')$), 
and put 
$$
A_{G} \ = \ R_{G} [[y_{e} \, (e \in E_{G})]], \ \ 
B_{G} \ = \ A_{G} \left[ \prod_{e \in E_{G}} y_{e}^{-1} \right]. 
$$ 
Then the ring $A_{G}$ is complete with respect to its ideal $I_{G}$ 
generated by $y_{e}$ $(e \in E_{G})$. 
 
In a similar way to \cite[Section 2]{I1}, 
we construct the universal Schottky group $\Gamma_{G}$ associated with $G$ as follows. 
For $e \in E_{G}$, 
denote by $\phi_{e}$ an element of $GL_{2}(B_{G})$ 
and hence of $PGL_{2}(B_{G}) = GL_{2}(B_{G})/B_{G}^{\times}$ defined as 
$$
\phi_{e} = \left( \begin{array}{cc} x_{e} & x_{-e} \\ 1 & 1 \end{array} \right) 
\left( \begin{array}{cc} 1 & 0 \\ 0 & y_{e} \end{array} \right) 
\left( \begin{array}{cc} x_{e} & x_{-e} \\ 1 & 1 \end{array} \right)^{-1} 
\quad {\rm mod} \left( B_{G}^{\times} \right) 
$$
which is equivalent to 
$$
\left[ \phi_{e}(z), z; x_{e}, x_{-e} \right] = y_{e} \quad (z \in \mathbb{P}^{1}), 
\eqno(3.1) 
$$
where $[a, b; c, d]$ denotes the cross ratio defined as 
$$
[a, b; c, d] = \frac{(a-c)(b-d)}{(a-d)(b-c)}. 
$$
For any reduced path $\rho = e_{1} \cdots e_{n}$ 
which is the product of oriented edges $e_{1}, ... , e_{n}$, 
one can associate an element $\rho^{*}$ of $PGL_{2}(B_{G})$ 
having reduced expression $\phi_{e_{n}} \cdots \phi_{e_{1}}$. 
Fix a base vertex $v_{\rm B}$ of $V_{G}$, 
and consider the fundamental group 
$\pi_{1} (G; v_{\rm B})$ which is a free group 
of rank $g = {\rm rank}_{\mathbb Z} H_{1}(G, {\mathbb Z})$. 
Then the correspondence $\rho \mapsto \rho^{*}$ 
gives an injective anti-homomorphism 
$\pi_{1} (G; v_{\rm B}) \rightarrow PGL_{2}(B_{G})$ 
whose image is denoted by $\Gamma_{G}$. 

It is shown in \cite[Section 3]{I1} that if a graph $G$ is of $(g, n)$-type, 
namely, ${\rm rank} \, H_{1}(G, \mathbb{Z}) = g$ and $|L| = n$, 
then there exists a family ${\mathcal R}_{G}$ of $n$-pointed complex curves of genus $g$ 
by taking $x_{h}$ $(h \in LE_{G})$ and $y_{e}$ $(e \in E_{G})$ as complex numbers such that 
$$
x_{e} \neq x_{-e} \ (e \in E_{G}), 
\quad x_{h} \neq x_{h'} \ (\text{$h \neq h'$ and $v(h) = v(h')$}), 
$$
and that $y_{e}$'s are sufficiently small. 
This family satisfies the following properties: 

\begin{itemize}

\item 
the closed fiber of ${\mathcal R}_{G}$ obtained by putting $y_{e} = 0$ $(e \in E_{G})$ 
becomes the degenerate pointed complex curve with dual graph $G$. 
More precisely, this curve is obtained from the collection of $P_{v} := \mathbb{CP}^{1}$ 
$(v \in V_{G})$ with marked points $x_{h}$ for $h \in L_{G}$ with $v(h) = v$ 
by identifying the points $x_{e} \in P_{v(e)}$ and $x_{-e} \in P_{v(-e)}$ ($e \in E_{G}$), 
where $x_{h}$ $(h \in LE_{G})$ are identified with the corresponding points on $P_{v(h)}$, 

\item 
the family ${\mathcal R}_{G}$ gives rise to a deformation of degenerate pointed curves with dual graph $G$, 
and any member of ${\mathcal R}_{G}$ for $y_{e} \neq 0$ $(e \in E_{G})$ is smooth, 
namely, a pointed Riemann surface and is uniformized by the Schottky group $\Gamma$ 
obtained from $\Gamma_{G}$ taking $x_{h}, y_{e}$ as complex numbers as above. 

\end{itemize}

\subsection{Abelian differentials}

For a graph $G$ of $(g, n)$-type, 
we define the universal Abelian differentials (of the first, third kinds) 
on the above family ${\mathcal R}_{G}$ of $n$-pointed complex curves of genus $g$ 
which are Schottky uniformized by $\Gamma$. 
Fix a set $\{ \gamma_{1},..., \gamma_{g} \}$ of generators of $\Gamma$, 
and for each $\gamma_{i}$, 
denote by $\alpha_{i}$ (resp. $\alpha_{-i}$) its attractive (resp. repulsive) fixed points, 
and by $\beta_{i}$ its multiplier which satisfy 
$$
\left[ \gamma_{i}(z), z; \alpha_{i}, \alpha_{-i} \right] = \beta_{i}. 
$$ 
Then for each $i = 1,..., g$, 
we define the associated {\it universal Abelian differential of the first kind} as 
$$
\varpi_{i} = 
\sum_{\gamma \in \Gamma / \left\langle \gamma_{i} \right\rangle} 
\left( \frac{1}{z - \gamma(\alpha_{i})} - \frac{1}{z - \gamma(\alpha_{-i})} \right) dz. 
$$
Take a maximal subtree $T_{G}$ of $G$, 
and for each $l \in L_{G}$, 
denote by $\rho_{l} = e_{1} \cdots e_{n}$ the unique path in $T_{G}$ from $v(l)$ to $v_{\rm B}$ 
and put $\rho_{l}^{*} = \phi_{e_{n}} \cdots \phi_{e_{1}}$. 
Then for each $l_{1}, l_{2} \in L_{G}$ with $l_{1} \neq l_{2}$, 
we define the associated {\it universal Abelian differential of the the third kind} as 
$$
\varpi_{l_{1}, l_{2}} = \sum_{\gamma \in \Gamma} 
\left( \frac{d \gamma(z)}{\gamma(z) - \rho_{l_{1}}^{*}(x_{l_{1}})} - 
\frac{d \gamma(z)}{\gamma(z) - \rho_{l_{2}}^{*}(x_{l_{2}})} \right). 
$$

A  (holomorphic or meromorphic) global section of the dualizing sheaf on a stable curve 
is called a {\it stable differential} (cf. \cite{DM}). 
Then the following assertions are shown in \cite[Section 3]{I2}: 
\begin{itemize}

\item 
the functions $\varpi_{i}/dz$ and $\varpi_{l_{1},l_{2}}/dz$ of $z$ are expanded 
as power series in $y_{e}$ $(e \in E_{G})$ with coefficients in 
$R_{G} \left[ z, \prod_{h \in LE_{G}} (z - x_{h})^{-1} \right]$, 

\item 
the set $\{ \varpi_{i} \}_{1 \leq i \leq g}$ gives a basis of relative stable differentials 
on $\mathcal{R}_{G}$ which are holomorphic, 

\item 
for each $l_{1}, l_{2} \in L_{G}$ such that $l_{1} \neq l_{2}$, 
$\varpi_{l_{1}, l_{2}}$ is a relative stable differential on $\mathcal{R}_{G}$  
whose poles are simple at $x_{l_{1}} \in P_{v(l_{1})}$ (resp. $x_{l_{2}} \in P_{v(l_{2})}$) 
with residues $1$ (resp. $-1$), 

\item 
the Jacobian of $\mathcal{R}_{G}$ for $y_{e} \neq 0$ $(e \in E_{G})$ becomes 
the family of Abelian varieties whose multiplicative periods are given in \cite[3.10]{I1} 
as $P_{ij} = \prod_{\gamma} \psi_{ij}(\gamma)$ $(1 \leq i, j \leq g)$, 
where $\gamma$ runs through all representatives of 
$\langle \gamma_{i} \rangle \backslash \Gamma / \langle \gamma_{j} \rangle$ 
and 
$$
\psi_{ij}(\gamma) = \left\{ \begin{array}{ll} 
\beta_{i} 
& (\mbox{$i = j$ and $\gamma \in \langle \gamma_{i} \rangle$}), 
\\
\left[ \alpha_{i}, \alpha_{-i}; \gamma(\alpha_{j}), \gamma(\alpha_{-j}) \right] 
& (\mbox{otherwise}). 
\end{array} \right. 
$$ 
Furthermore, $P_{ij}$ are expanded as elements of $B_{G}$, 

\item
take a canonical (first) homology basis $\{ a_{i}, b_{i} \}_{1 \leq i \leq g}$ in $\mathcal{R}_{G}$ 
such that each $b_{i}$ corresponds to $\gamma_{i}^{-1}$ and 
each $a_{i}$ goes around an edge in $G \setminus T_{G}$.    
Then 
$$
\oint_{a_{i}} \varpi_{j} = 2 \pi \sqrt{-1} \delta_{ij}, 
\quad \exp \left( \oint_{b_{i}} \varpi_{j} \right) = P_{ij}, 
$$  
where $\delta_{ij}$ denotes the Kronecker delta.

\end{itemize}

\section{General tropical limits of harmonic amoebas} 

Let $C = (G, \ell)$ be a tropical curve, 
and denote by $\mathcal{R}_{G}$ the family of $n$-pointed complex curves of genus $g$ 
associated with $G$, 
where $G$ is of $(g, n)$-type. 
Recall that $\mathcal{R}_{G}$ gives deformations of a degenerate complex curve with dual graph $G$ 
by the parameters $y_{e}$ $(e \in E)$. 
For sufficiently small positive numbers $s$, 
we consider the deformations $\{ \mathcal{R}_{s} \}_{s>0}$ satisfying $|y_{e}| = s^{\ell(e)}$. 
For each $e \in E_{G}$, 
take an open neighborhood $U_{e}$ of $x_{e}$ in $P_{v(e)}$ such that 
$\phi_{e}$ maps the closure of $U_{-e}$ onto $P_{v(e)} \setminus U_{e}$, 
and for $v \in V_{G}$, 
put $F_{v} = P_{v} \setminus \bigcup_{e \in E, v(e) = v} U_{e}$. 
Then the domain $\Omega_{\Gamma}$ of discontinuity for $\Gamma$ is obtained as 
$$
\Omega_{\Gamma} = \bigcup_{\rho} \varphi_{\rho}(F_{v}) \subset P_{v_{\rm B}}, 
$$
where $\rho$ runs through paths from any $v \in V_{G}$ to $v_{\rm B}$, 
and hence $\mathcal{R}_{s}$ is expressed as the quotient space 
$\Omega_{\Gamma}/\Gamma$. 

In what follows, 
fix a maximal subtree $T_{G}$ of $G$, 
and for $l_{1}, l_{2} \in L_{G}$ with $l_{1} \neq l_{2}$, 
denote by $\rho_{l_{1},l_{2}}$ the unique path in the tree $T_{G}$ from $l_{1}$ to $l_{2}$. 

\subsection{Imaginary normalized differentials}

Let $\varpi_{i}$ ($1 \leq i \leq g$), 
$\varpi_{l_{1},l_{2}}$ ($l_{1}, l_{2} \in L_{G}$ with $l_{1} \neq l_{2}$) 
be the above universal Abelian differentials on $\mathcal{R}_{s}$. 

\begin{proposition} 

{\rm (1)} 
For each oriented edge $h \in E_{G}$, 
\begin{eqnarray*}
\varpi_{i} & = & {\rm w}^{b_{i}}(h) 
\left( \frac{1}{z - x_{h}} - \frac{1}{z - x_{-h}} \right) dz + \left( \sum_{e \in E_{G}} y_{e} \, O(1) \right) dz, 
\\ 
\varpi_{l_{1},l_{2}} & = & {\rm w}^{\rho_{l_{1},l_{2}}}(h) 
\left( \frac{1}{z - x_{h}} - \frac{1}{z - x_{-h}} \right) dz + \left( \sum_{e \in E_{G}} y_{e} \, O(1) \right) dz 
\end{eqnarray*}
around $z = x_{h}, x_{-h}$. 

{\rm (2)}
For a point $a \in \Omega_{\Gamma}$ and each oriented edge $h \in E_{G}$,  
\begin{eqnarray*}
\int_{a}^{\phi_{h}(a)} \varpi_{i} 
& = & \log (y_{h}) \, {\rm w}^{b_{i}}(h) + \sum_{e \in E_{G}} y_{e} \, O(1), 
\\
\int_{a}^{\phi_{h}(a)} \varpi_{l_{1},l_{2}} 
& = & \log (y_{h}) \, {\rm w}^{\rho_{l_{1},l_{2}}}(h) + \sum_{e \in E_{G}} y_{e} \, O(1). 
\end{eqnarray*}

{\rm (3)}
For a point $a \in \Omega_{\Gamma}$ and each oriented edge $h \in E_{G}$,  
\begin{eqnarray*}
{\rm Re} \left( \int_{a}^{\phi_{h}(a)} \varpi_{i} \right) 
& = & 
(\log s) \ell(h) {\rm w}^{b_{i}}(h) + O(1), 
\\ 
{\rm Re} \left( \int_{a}^{\phi_{h}(a)} \varpi_{l_{1},l_{2}} \right) 
& = & 
(\log s) \ell(h) {\rm w}^{\rho_{l_{1},l_{2}}}(h) + O(1). 
\end{eqnarray*}

\end{proposition}

\begin{proof}
The assertion (1) is shown in \cite[Theorem 3.3]{I2}, 
(2) follows from (1) and (3.1), 
and (3) follows from (2) since ${\rm w}^{b_{i}}(h), {\rm w}^{\rho_{l_{1},l_{2}}}(h) \in \mathbb{R}$. 
\end{proof}

\begin{corollary}
One has 
$$
{\rm Re} \left( \oint_{b_{i}} \varpi_{j} \right) 
= \log s \int_{b_{i}} {\rm w}^{b_{j}} + O(1), 
\quad 
{\rm Re} \left( \oint_{b_{i}} \varpi_{l_{1},l_{2}} \right) 
= \log s \int_{b_{i}} {\rm w}^{\rho_{l_{1},l_{2}}}(h) + O(1). 
$$
\end{corollary}

\begin{theorem}
Denote by $\omega_{l_{1},l_{2}}$ the unique imaginary normalized differential on $\mathcal{R}_{s}$ 
whose poles are simple at $\varphi_{l_{1}}(x_{l_{1}}), \varphi_{l_{2}}(x_{l_{2}})$ with residues $1, -1$ respectively. 
Then one has 
$$
\omega_{l_{1},l_{2}} = \varpi_{l_{1},l_{2}} - \sum_{i=1}^{g}  
{\rm Re} \left( \oint_{b_{i}} \varpi_{l_{1},l_{2}} \right)_{1 \leq i \leq g} 
\left( {\rm Re} \left( \oint_{b_{j}} \varpi_{i} \right)_{1 \leq i,j \leq g} \right)^{-1} (\varpi_{i})_{i}^{T}. 
$$
\end{theorem}

\begin{proof}
Since the line integral on $\mathcal{R}_{s}$ around $a_{i}$ is a loop integral on the domain 
$\Omega_{\Gamma} \subset \mathbb{CP}^{1}$, 
$$
{\rm Re} \left( \oint_{a_{i}} \varpi_{j} \right) = {\rm Re} \left( \oint_{a_{i}} \varpi_{l_{1},l_{2}} \right) = 0. 
$$ 
Then 
$$
\omega'_{l_{1},l_{2}} = \varpi_{l_{1},l_{2}} - \sum_{i=1}^{g}  
{\rm Re} \left( \oint_{b_{i}} \varpi_{l_{1},l_{2}} \right)_{i} 
\left( {\rm Re} \left( \oint_{b_{j}} \varpi_{i} \right)_{i,j} \right)^{-1} (\varpi_{i})_{i}^{T}
$$
satisfies 
$$
{\rm Re} \left( \oint_{a_{i}} \omega'_{l_{1},l_{2}} \right) = 
{\rm Re} \left( \oint_{b_{i}} \omega'_{l_{1},l_{2}} \right) = 0, 
$$
and has simple poles at $\varphi_{l_{1}}(x_{l_{1}}), \varphi_{l_{2}}(x_{l_{2}})$ with residues $1, -1$ respectively. 
Therefore, one has $\omega'_{l_{1},l_{2}} = \omega_{l_{1},l_{2}}$. 
\end{proof}

\begin{corollary}
The $g \times g$ real matrix $\left( \oint_{b_{j}} {\rm w}^{b_{i}} \right)_{i,j}$ is 
symmetric and invertible, 
and one has 
$$
\omega_{l_{1},l_{2}} = \varpi_{l_{1},l_{2}} - \sum_{i=1}^{g} \left( \oint_{b_{i}} {\rm w}^{\rho_{l_{1},l_{2}}} \right)_{i} 
\left( \left( \oint_{b_{j}} {\rm w}^{b_{i}} \right)_{i,j} \right)^{-1} (\varpi_{i})_{i}^{T} 
+ \left( \frac{O(1)}{\log s} \right) dz. 
$$
\end{corollary}

\begin{proof}
The matrix $\left( \oint_{b_{j}} {\rm w}^{b_{i}} \right)_{i,j}$ is symmetric and positive definite, 
and hence is invertible. 
Therefore, 
the formula follows from Theorem 4.3 and Corollary 4.2. 
\end{proof}

\begin{proposition}
Denote by ${\rm w}_{l_{1},l_{2}}$ the exact $1$-form on $C$ whose residues are $1, -1$ 
at $l_{1}, l_{2}$ respectively. 
Then 
$$
{\rm w}_{l_{1},l_{2}} = {\rm w}^{\rho_{l_{1},l_{2}}} - \sum_{i=1}^{g}  
\left( \oint_{b_{i}} {\rm w}^{\rho_{l_{1},l_{2}}} \right)_{i} 
\left( \left( \oint_{b_{j}} {\rm w}^{b_{i}} \right)_{i,j} \right)^{-1} ({\rm w}^{b_{i}})_{i}^{T}. 
$$
\end{proposition}

\begin{proof}
The formula follows from that the left and right hand sides are both exact $1$-forms 
on $C$ whose residues are $1, -1$ at $l_{1}, l_{2}$ respectively. 
\end{proof}

\begin{corollary}
For each oriented edge $h \in E$, 
$$
\omega_{l_{1},l_{2}} = {\rm w}_{l_{1},l_{2}}(h) 
\left( \frac{1}{z - x_{h}} - \frac{1}{z - x_{-h}} \right) dz + 
\left( \frac{O(1)}{\log s} + \sum_{e \in E} y_{e} \, O(1) \right) dz
$$
around $z = x_{h}, x_{-h}$. 
\end{corollary}

\begin{proof}
The assertion follows from Propositions 4.1 (1), Corollary 4.4 and Proposition 4.5. 
\end{proof}

\subsection{Tropical convergence of integrals}

\begin{theorem}
For a point $a \in \Omega_{\Gamma}$ and each oriented edge $h \in E_{G}$,  
$$
\int_{a}^{\phi_{h}(a)} \omega_{l_{1},l_{2}} = 
{\rm w}_{l_{1},l_{2}}(h) \log (y_{h}) + \frac{O(1)}{\log s} + \sum_{e \in E} y_{e} \, O(1). 
$$ 
\end{theorem}

\begin{proof}
The formula follows from Corollary 4.6 and (3.1). 
\end{proof}

\subsection{Proof of Theorem 1.1} 
Since the maps $\mathcal{A}_{R}/\log s$ and $\pi_{R}$ are both linear in $R$, 
one may assume that $\omega_{R} = \omega_{l_{1},l_{2}}$. 
Take an oriented edge $e \in E_{G}$ and a point $p \in P_{v(-e)} \setminus U_{-e}$. 
Then for any point $q \in P_{v(e)} - U_{e}$, 
$$
\lim_{s \downarrow 0} \frac{1}{\log s} \int_{p}^{q} \omega_{l_{1},l_{2}} = 
\lim_{s \downarrow 0} \frac{1}{\log s} \int_{p}^{\phi_{e}(p)} \omega_{l_{1},l_{2}}. 
$$
For a constant $c$ and $0 \leq t \leq 1$, 
put 
$$
\phi_{c,t} = \left( \begin{array}{cc} x_{e} & x_{-e} \\ 1 & 1 \end{array} \right) 
\left( \begin{array}{cc} 1 & 0 \\ 0 & c \cdot s^{t \cdot l(e)} \end{array} \right) 
\left( \begin{array}{cc} x_{e} & x_{-e} \\ 1 & 1 \end{array} \right)^{-1}. 
$$
Then by Theorem 4.7, 
\begin{eqnarray*}
\lim_{s \downarrow 0} \frac{1}{\log s} \int_{p}^{\phi_{c,t}(p)} \omega_{l_{1},l_{2}} 
& = & 
\lim_{s \downarrow 0} \frac{1}{\log s} {\rm w}_{l_{1},l_{2}}(e) 
\log \left[ \phi_{c,t}(p), p; x_{e}, x_{-e} \right] 
\\
& = & 
t \, l(e) {\rm w}_{l_{1},l_{2}}(e) 
\\
& = &t \int_{e} {\rm w}_{l_{1},l_{2}}.  
\end{eqnarray*}
This completes the proof of Theorem 1.1. 

\subsection{Proof of Theorem 1.2} 
One may assume that $\omega_{R} = \omega_{l_{1},l_{2}}$, 
$\ell_{s}(e) = \ell(e)$ and $\theta_{s}(e)$ is independent of $s$, 
and put $\theta(e) = \theta_{s}(e)$. 
Then for a point $a \in \Omega_{\Gamma}$, 
by Theorem 4.7, 
\begin{eqnarray*}
{\rm Re} \left( \oint_{\rho_{s}} \omega_{l_{1},l_{2}} \right)  
& = & 
\log s \, \sum_{j=1}^{n} {\rm w}_{l_{1},l_{2}}(e_{j}) \ell(e_{j}) 
+ \frac{O(1)}{\log s} + \sum_{e \in E_{G}} y_{e} \, O(1) 
\\
& = & 
\log s \int_{\rho} {\rm w}_{l_{1},l_{2}} + \frac{O(1)}{\log s} + \sum_{e \in E_{G}} y_{e} \, O(1) 
\\
& \rightarrow & 
0
\end{eqnarray*}
under $s \downarrow 0$ since ${\rm w}_{l_{1},l_{2}}$ is exact. 
Therefore, by Theorem 4.7 again, 
\begin{eqnarray*}
\lim_{s \downarrow 0} \oint_{\rho_{s}} \omega_{l_{1},l_{2}} 
& = & 
\lim_{s \downarrow 0} \sqrt{-1} \, {\rm Im} \int_{a}^{\rho^{*}(a)} \omega_{l_{1},l_{2}} 
\\
& = & 
\lim_{s \downarrow 0} \sqrt{-1} \, {\rm Im} 
\left( \sum_{j=1}^{n} \int_{\rho_{j-1}^{*}(a)}^{\rho_{j}^{*}(a)} \omega_{l_{1},l_{2}} \right) 
\quad (\rho_{j}^{*} := \phi_{e_{j}} \cdots \phi_{e_{1}}) 
\\
& = & 
\sum_{j=1}^{n} {\rm w}_{l_{1}.l_{2}}(e_{j}) \log(\theta(e_{j})).  
\end{eqnarray*}
This completes the proof of Theorem 1.2.





\section*{Declarations}

\begin{itemize}

\item {\bf Funding}. 
This work is partially supported by the JSPS Grant-in-Aid for 
Scientific Research No. 25K06920.

\item {\bf Conflict of interest statement}. 
The author declares that he has no known competing financial interests 
or personal relationships that could have appeared to influence the work 
reported in this article. 


\item {\bf Data availability}. 
No data was used for the research described in this article. 

\end{itemize}

\renewcommand{\refname}{\bf References}

\end{document}